\documentclass[11pt,reqno]{article}

\usepackage{amsfonts,color,amsmath,amssymb,fancyhdr}
\usepackage{amsthm}
\usepackage{tikz}
\usepackage{graphicx}
\usepackage{subfigure}
\usepackage{mathrsfs}
\usepackage{indentfirst}
\usepackage[top=1in,bottom=1in,left=1in,right=1in]{geometry}

\usepackage[colorlinks,linkcolor=blue, anchorcolor=red,citecolor=green]{hyperref}

\def\Box{\vcenter{\vbox{\hrule\hbox{\vrule
     \vbox to 8.8pt{\hbox to 10pt{}\vfill}\vrule}\hrule}}}
\renewenvironment{proof}{{\noindent\it Proof.}\ }{\hfill $\blacksquare$\par\medskip}

\newcommand{\tr}{\textup{Tr}}

\newcommand{\ra}{\rangle}

\newcommand{\la}{\langle}

\newcommand{\F}{{\mathbb F}}
\newcommand{\Z}{{\mathbb Z}}

\newcommand{\cQ}{{\mathcal Q}}

\newcommand{\cP}{{\mathcal P}}

\newcommand{\GU}{\textup{GU}}

\newcommand{\PG}{\textup{PG}}
\newcommand{\PGU}{\textup{PGU}}

\newcommand{\Aut}{\textup{Aut}}

\newcommand{\PGaU}{\textup{P}\Gamma\textup{U}}

\newtheorem{thm}{Theorem}[section]
\newtheorem{lemma}[thm]{Lemma}
\newtheorem{corollary}[thm]{Corollary}

\newtheorem{example}[thm]{Example}
\numberwithin{equation}{section}
\numberwithin{thm}{section}

\newtheorem{remark}[thm]{Remark}

\begin{document}
\newcommand{\stopthm}{\begin{flushright}
\(\box \;\;\;\;\;\;\;\;\;\; \)
\end{flushright}}
\newcommand{\symfont}{\fam \mathfam}

\title{On transitive ovoids of finite Hermitian polar spaces}
\author{Tao Feng \and Weicong Li\thanks{Corresponding author: conglw@zju.edu.cn}}
\date{}
%
\maketitle
\begin{abstract}
In this paper, we complete the classification of transitive ovoids of finite Hermitian polar spaces.
\newline

\noindent\text{Keywords:} transitive ovoids, Hermitian polar spaces, Singer orbits.

\end{abstract}


\section{Introduction}

 An \textit{ovoid} $O$ of a finite classical polar space $\cP$ of rank $r>1$ is a subset of points  of $\cP$ that has exactly one  point in common with each maximal totally singular subspace of $\cP$. The size of an ovoid, if one exists, is called the \textit{ovoid number}, which we denote by $\theta(\cP)$. Alternatively, an ovoid of $\cP$ is a set of $\theta(\cP)$ points of $\cP$ no two of which are perpendicular. There has been extensive research on ovoids due to their close connection with various geometric objects as well as other branches of combinatorics, cf. \cite{Thas1981,Thas2001}. There have been two major generalizations of the concept of ovoids in the past two decades, namely $m$-systems and intriguing sets. Please refer to \cite[Chapter 7]{GGG} and \cite{BambergKLP2007} for more details and the references therein.

Let $V$ be an $(n+1)$-dimensional vector space over $\F_{q^2}$, equipped with a nondegenerate Hermitian form $h:\,V\times V\rightarrow\F_{q^2}$. Let $\PG(V)$ be the projective space associated to $V$, and denote by $\la v\ra$ the projective point that corresponds to the $1$-dimensional vector space spanned by a nonzero vector $v$. The \textit{Hermitian polar space} $H(n,q^2)$ associated to the pair $(V,\,h)$ consists of the totally singular subspaces of $V$ with respect to $h$, with inclusion as incidence relation. The maximal totally singular subspaces of $H(n,q^2)$ have dimension $r=\lfloor (n+1)/2\rfloor$, and $H(n,q^2)$ has rank $r$. In the case $n>2$ is even, $H(n,q^2)$ has no ovoids, cf. \cite{Thas1981}.
In the case $n\ge 3$ is odd,  there are many ovoids in $H(3,q^2)$, but for odd $n$ with $n>3$ little is known about the existence of ovoids in $H(n,q^2)$ besides the results in \cite{BM1995,DeM2005,Moor1996}.  Let $n$ be odd and $O$ be an ovoid of $H(n,q^2)$. Then $O$ has size $q^n+1$. For a point $P$ of $H(n,q^2)$,  we have
\begin{equation}\label{eqn_def}
 |P^{\perp}\cap O|=\Big\lbrace\begin{array}{lcl}
 1,\, &\text{if}& P\in O\\
 q^{n-2}+1,\, &\text{if}& P\not\in O
 \end{array},
\end{equation}
by \cite{BambergKLP2007}. Here, $\perp$ is the polarity associated with $H(n,q^2)$. Moreover, it is clear that $g(O)$ is also an ovoid for each $g \in\PGaU(n+1,q^2)$; we say that $g(O)$ is \textit{projectively equivalent} to $O$.

All the known ovoids of Hermitian polar spaces of rank $r>1$ live in $H(3,q^2)$. A nondegenerate plane section of $H(3,q^2)$ is a \textit{classical} ovoid, and all classical ovoids of $H(3,q^2)$ are projectively equivalent. There is a powerful way to obtain new ovoids from an old ovoid of $H(3,q^2)$, known as \textit{derivation}, due to Payne and Thas \cite{Thas1994Spreads}. To be specific,
given an ovoid $O$ of $H(3,q^2)$ such that there is  a line $\ell$ of $\PG(3,q^2)$ that intersects $O$ in $q+1$ points, let $O'$ be the set obtained from $O$ by removing the points on $\ell$ and then adding the singular points on $\ell^\perp$. Then $O'$ is also an ovoid.  The Singer-type ovoids with $q$ even, which we will describe later, can be obtained from the classical ovoids by multiple derivations.

An ovoid of $H(n,q^2)$ is \textit{transitive} if there is a subgroup $H$ of $\PGaU(n+1,q^2)$ that stabilizes and acts transitively on the points of $O$. If $H$ is a subgroup of $\PG\textup{U}(n+1,q^2)$, then it is said to be \textit{linearly} transitive. In $H(3,q^2)$, the classical ovoids are linearly transitive ovoids, whose full stabilizer has a nonabelian composition factor $\textup{PSU}(3,q^2)$. When $q>2$ is even, Cossidente and Korchm\'{a}ros constructed a new family of transitive ovoids of $H(3,q^2)$, the Singer-type ovoids, based on the cyclic spreads of $H(2,q^2)$ in \cite{BEKS1993}. Such ovoids have a soluble
stabilizer in $\PGaU(4,q^2)$.  In the same paper, they classified the linearly transitive ovoids of $H(3,\,q^2)$ with $q$ even, and the only such ovoids are the classical ovoids and the Singer-type ovoids. The smallest case $q=2$ is due to Brouwer and Wilbrink \cite{Brouwer1990}, who classified all ovoids of $H(3,\,2^2)$.
In 2009, Bamberg and Penttila \cite{Bamberg2009Classificationovoid} classified transitive ovoids of finite classical polar spaces  admitting an insoluble transitive automorphism group, which is based on the deep results in \cite{GPPS1999}. There are no such transitive ovoids in $H(n,q^2)$ with $n>3$. Besides the classical ovoids, the only such ovoids in $H(3,q^2)$ exist when $q=5$. We refer the reader to \cite{Bamberg2009Classificationovoid} for a survey of classification results of transitive ovoids in other classical polar spaces.

In this paper, we give a complete classification of transitive ovoids in $H(3,q^2)$, and show that for a given odd integer $n\ge 5$ there is no transitive ovoid in $H(n,q^2)$. To be specific, we prove the following theorem.
\begin{thm}\label{main}	
Let $q$ be a prime power and $n\ge 3$ be an integer. Suppose that $O$ is a transitive ovoid of $H(n,q^2)$. Then $n=3$, and $O$ is projectively equivalent to one of the following:
\begin{enumerate}
\item[(1)] a classical ovoid $H(2,q^2)$, with full stabilizer $\Gamma \textup{U}(3,q^2)$;
\item[(2)] a Singer-type ovoid for even $q=2^d$, with full stabilizer $\Z_{q^3+1}:\,6d$ if $q>2$;
\item[(3)] an exceptional ovoid of $H(3,5^2)$, with full stabilizer $\Z_2\times (\Z_3\times \textup{PSL}(2,7)):2$;
\item[(4)] an ovoid of $H(3,8^2)$, with full stabilizer $\Z_{57}:9$;
\item[(5)] an ovoid of $H(3,8^2)$, with full stabilizer $\Z_{57}:18$.
\end{enumerate}
\end{thm}
\begin{remark}
Please refer to \cite{Bamberg2009Classificationovoid} for a description of the exceptional ovoid in $H(3,5^2)$, which is based on a unital spread of $H(2,5^2)$ discovered in \cite{CT2005}. For a description of the Singer-type ovoids and the two exceptional ovoids in $H(3,8^2)$, please refer to Example \ref{rem_Singer} and Example \ref{rem_q8} respectively.
\end{remark}

This paper is organized as follows. In Section 2, we first present some preliminary results. In Section 3, we introduce the model of $H(n,q^2)$, $n$ odd, that we shall use in this paper. We study the orbits of certain Singer groups of order $q^n+1$ and determine when they form ovoids of $H(n,q^2)$. In particular, in the case $n=3$ and $q$ is even we give an algebraic description of the Singer-type ovoids. In Section 4, we first derive some restrictions on the parameters of a transitive ovoid with a soluble stabilizer in $\PGaU(n+1,q^2)$, and then give the proof of Theorem \ref{main}.

\section{Preliminaries}

\subsection{Some technical lemmas}
Throughout this paper, let $p$ be a prime, $d$ be a positive integer and $n$ be an odd positive integer greater than or equal to $3$. Set $q=p^d$, take $\omega_0$ to be a fixed element of order $(q^n+1)(q-1)$ in $\F_{q^{2n}}$, and write $\omega=\omega_0^{q-1}$. If $E$ is a field extension of $F$ of finite degree, we use $\tr_{E/F}$ for the relative trace function from $E$ to $F$.

\begin{lemma}\label{lem_xomega}
Take notation as above, and set $W:=\F_{q^n}^*\cdot \la \omega\ra$. Then $\F_{q^{2n}}^*=W$ if $q$ is even, and $\F_{q^{2n}}^*=W\cup W\omega_0$ if $q$ is odd.
\end{lemma}
\begin{proof}
In the case $q$ is odd, $\gcd(q^n+1,q^n-1)=2$, so $W$ is a multiplicative subgroup of index $2$ in $\F_{q^{2n}}^*$; moreover, $\omega_0$ is a nonsquare of $\F_{q^{2n}}^*$ since $\frac{q^{2n}-1}{(q^n+1)(q-1)}=q^{n-1}+q^{n-2}+\cdots+1$ is odd. The claim now follows readily. In the case $q$ is even, the claim follows from the fact $\gcd(q^n+1,q^n-1)=1$.
\end{proof}
\medskip

\begin{lemma}\label{lem_wiproperty}
Take notation as above, and assume that $n=3$. Let $x$ be an element of $\la \omega\ra\setminus\F_{q^2}$. Then $\tr_{\F_{q^6}/\F_{q^2}}(x) \ne 1$.
\end{lemma}
\begin{proof}
Suppose to the contrary that $x$ is an element of $\la \omega\ra\setminus\F_{q^2}$ such that $\tr_{\F_{q^6}/\F_{q^2}}(x)=1$. Write $s:=x^{1+q^2+q^4}$, which lies in $\F_{q^2}^*$. We compute that
\begin{align*} \tr_{\F_{q^6}/\F_{q^2}}(x^{q^2+q^4})=s\tr_{\F_{q^6}/\F_{q^2}}(x^{-1})=s\tr_{\F_{q^6}/\F_{q^2}}(x^{q^3})=s.
\end{align*}
Since $x\not\in\F_{q^2}$, the minimal polynomial of $x$ over $\F_{q^2}$ is
\begin{align*} (X-x)(X-x^{q^2})(X-x^{q^4})&=X^3-\tr_{\F_{q^6}/\F_{q^2}}(x)X^2+\tr_{\F_{q^6}/\F_{q^2}}(x^{q^2+q^4})X-s\\
	&=X^3-X^2+sX-s.
\end{align*}
However, this polynomial factors as $(X-1)(X^2+s)$: a contradiction.
This completes the proof.
\end{proof}
\medskip

\begin{lemma}\label{lem_ind}
Suppose that $\gcd(n,e)=1$. Then $\tr_{\F_{q^{ne}}/\F_{q^e}}(x)=\tr_{\F_{q^{n}}/\F_q}(x)$ for $x\in\F_{q^n}$, and a basis  of $\F_{q^e}$ over $\F_q$ is also a basis of $\F_{q^{ne}}$ over $\F_{q^n}$.
\end{lemma}

\begin{proof}
Take $x\in\F_{q^n}$, so that $x^{q^n}=x$. Since $\gcd(n,e)=1$, we have $\{ie\pmod{n}:\,0\le i\le n-1\}=\{0,\cdots,n-1\}$, so
\begin{align*}
\tr_{\F_{q^{ne}}/\F_{q^e}}(x)&=x+x^{q^e}+\cdots+x^{q^{e(n-1)}}\\
&=x+x^q+\cdots+x^{q^{n-1}}=\tr_{\F_{q^{n}}/\F_q}(x).
\end{align*}

Let $\zeta_1,\cdots,\zeta_e$ be a basis of $\F_{q^e}$ over $\F_q$. Suppose that  $\sum_{i=1}^ec_i\zeta_i=0$ for some elements $c_i\in\F_{q^n}$. For any  $x\in\F_{q^n}$, we have
\begin{align*}
0&=\tr_{\F_{q^{ne}}/\F_{q^e}}\left(x\cdot\sum_{i=1}^ec_i\zeta_i\right)
=\sum_{i=1}^e\tr_{\F_{q^{n}}/\F_{q}}(c_ix)\zeta_i
\end{align*}
by the previous claim. Since the $\zeta_i$'s form a basis over $\F_{q}$, it follows that $\tr_{\F_{q^{n}}/\F_{q}}(c_ix)=0$ for each $i$. This holds for any $x\in\F_{q^n}$, so we have $c_i=0$ for each $i$. This proves the second claim.
\end{proof}

\medskip

\subsection{The bound of Blokhuis and Moorhouse}
For an positive integer $n$ and a prime $p$, define
\begin{equation}\label{eqn_F}
F(n,p):=\frac{1}{p^n}\binom{n+p-1}{n}^2-\frac{1}{p^n}\binom{n+p-2}{n}^2.
\end{equation}
\begin{thm}\cite[Corollary 1.3]{Moor1996}\label{thm_moo}
Let $q=p^d$ be a prime power with $p$ prime and $n$ be an odd integer. If $H(n,q^2)$ contains an ovoid, then $F(n,p)\ge1$.
\end{thm}
Since $\binom{n+p-1}{n}=\frac{n+p-1}{n}\binom{n+p-2}{n-1}$, $\binom{n+p-2}{n}=\frac{p-1}{n}\binom{n+p-2}{n-1}$,  we have
\[
F(n,p)=\binom{n+p-2}{n-1}^2 \frac{n+2p-2}{np^n}.
\]
We have the following \textit{Stirling's approximation} \cite{Stirling1955}:
\begin{equation}\label{eqn_stirling}
	n!=(ne^{-1})^n\sqrt{2\pi n}e^{\alpha_n},\quad \frac{1}{12n+1}< \alpha_n< \frac{1}{12n}.
\end{equation}
\begin{lemma}\label{lem_Fprop}
Let $n$ be a positive integer and $p$ be a prime. Let $F(n,p)$ be the function defined in  \eqref{eqn_F}.
\begin{enumerate}
\item[(1)] $F(n+1,p)<F(n,p)$ if either $p>3$ and $n\geq (p+1)/2$, or $p\le 3$ and $n\ge p+1$;
\item[(2)] For a prime $p>45$, we have $F((p+1)/2,p)<1$.
\end{enumerate}
\end{lemma}
\begin{proof}
Since $\binom{n+p-1}{n}=\frac{n+p-1}{n}\binom{n+p-2}{n-1}$, we have
\[
\begin{array}{rl}
	&\frac{n^2(n+1)p^{n+1}}{\binom{n+p-2}{n-1}^2 }\cdot(F(n,p)-F(n+1,p))\\
	=&pn(n+1)\cdot(n+2p-2)-n^2\cdot\frac{(n+p-1)^2}{n^2}\cdot(n+2p-1)\\
	=&(p-1)\cdot(n^3+(2p-3)n^2-3(p-1)n-2p^2+3p-1).
	\end{array}
\]
Write $f_p(n):=n^3+(2p-3)n^2-3(p-1)n-2p^2+3p-1$. Hence $F(n,p)-F(n+1,p)>0$ if and only if $f_p(n)>0$. Let $p$ be fixed. Since
\[
f_p'(n)=3n^2+(4p-6)n-3(p-1)=n^2+2n(n-1)+(p-1)(4n-3)>0,
\]
the function $f_p(n)$ is increasing in $n$ when $n\geq 1$. Here, we regard $f_p(n)$ as a polynomial in the variable $n$. We directly compute that
\[
f_p\left(\frac{p+1}{2}\right)=\frac{1}{8}(p-1)(5p^2-18p+1),
\]
which is positive for $p>3$. For $p=2,3$, we have $f_p(p+1)=p(3p^2-p+2)>0$. This proves the first claim.

We have $\frac{p+1}{2}+p-2=\frac{3}{2}(p-1)$. By \eqref{eqn_stirling}, we have
\begin{align*}
\binom{3(p-1)/2}{(p-1)/2}&= \frac{\left(\frac{3}{2}(p-1)\right)!}{\left(\frac{p-1}{2}\right)! \cdot (p-1)!}=\frac{3^{(3p-2)/2}}{2^{p-1}\sqrt{2\pi(p-1)}} e^{\beta}
\end{align*}
where $\beta=\alpha_{3(p-1)/2}-\alpha_{(p-1)/2}-\alpha_{(p-1)}$. Since $\frac{1}{12n+1}<\alpha_n<\frac{1}{12n}$, we have
\[
\beta< \frac{1}{18(p-1)} -\frac{1}{12(p-1)+1}-\frac{1}{6(p-1)+1}<0.
\]
We compute that
\[
F((p+1)/2,p)=\frac{3^{3p-2}}{4^{p-1}\cdot 2\pi(p-1)} \cdot \frac{5p-3}{p^{(p+1)/2}(p+1)}e^{2\beta}
  =\left(\frac{6.75^2}{p}\right)^{(p-1)/2}\frac{3(5p-3)}{2 \pi p(p^2-1)}e^{2\beta}.
\]
Since $6.75^2=45.5625$, $3(5p-3)<15p<2 \pi p(p^2-1)$ and $\beta<0$, we deduce that $F((p+1)/2,p)<1$ for $p>45$ as desired. This completes the proof.
\end{proof}

\medskip

\subsection{Primitive divisors}\label{subsec_pd}

For integers $x$ and $k$, with $x,\,k\ge 2$, a \textit{primitive} prime divisor of $x^k-1$ is a prime divisor $r$ of $x^k-1$ which does not divide $x^{k'}-1$ for all $1\le k'<k$. In other words, $x$ has order $k$ modulo $r$. As a corollary, we have $r\equiv 1\pmod{k}$.  By a theorem of Zsigmondy \cite{Z1892}, there is at least one primitive prime divisor of $x^k-1$ unless $(x,k)=(2,6)$ or $k=2$ and $x+1$ is a power of $2$. We denote such a prime by $x_k$. Note that if $(x,k)\ne (2,3)$, then $x^k+1$ is divisible by a primitive prime divisor $x_{2k}$ of $x^{2k}-1$.

The primitive prime divisors play an important role in the study of geometric objects that satisfy certain transitivity conditions, see \cite{BT2006,Bamberg2009Classificationovoid} for instance.  In the situation of the present paper, we have stronger information than the existence of primitive prime divisors. A divisor $r$ of $x^k-1$ that is relatively prime to each $x^i-1$ for $1\le i\le k$ is said to be a \textit{primitive divisor}, and we call the largest primitive divisor $\Phi^*_k(x)$ of $x^k-1$ the \textit{primitive part}. It is clear that each prime divisor of $\Phi^*_k(x)$ is a primitive prime divisor, so in particular $\Phi^*_k(x)\equiv 1\pmod{k}$ provided that $\Phi^*_k(x)>1$; conversely, each primitive prime divisor of $x^k-1$ divides $\Phi^*_k(x)$.

Let $n$ be an odd integer with $n\ge 3$, and let $q=p^d$ with $p$ a prime. Suppose that $O$ is a transitive ovoid of the Hermitian polar space $H(n,q^2)$ with ambient vector space $V$, and let $H$ be its stabilizer in $\PGaU(n+1,q^2)$.  In the case $(q,n)=(2,3)$, $O$ is either a classical ovoid or a Singer-type ovoid by \cite{Brouwer1990}. \textit{We assume that $(q,n)\ne (2,3)$ for the rest of the paper.} In particular, we always have $\Phi_{2nd}^*(p)>1$. Since $\Phi_{2nd}^*(p)$ is congruent to $1$ modulo $2nd$, we have $\gcd(2nd,\,\Phi_{2nd}^*(p))=1$.  It follows that $H\cap \PGU(n+1,q^2)$ has order divisible by $\Phi_{2nd}^*(p)$. Since the case $H$ is insoluble has been settled in \cite[Theorem 4.3]{Bamberg2009Classificationovoid}, we only consider the case where $H$ is soluble.  We shall need the following soluble case of the main theorem in \cite{BT2008Overgroup}, whose proof is based on \cite{GPPS1999}.
\begin{thm}\cite[Theorem 4.2]{BT2008Overgroup}\label{thm_BT}
Let $p$ be a prime, $q=p^d$, and $n$ be an odd integer such that $n\ge 3$, $(p,nd)\ne (2,3)$. Let $V$ be an $(n+1)$-dimensional vector space over $\F_{q^2}$, and let $h$ be a nondegenerate Hermitian form on $V$ with linear isometry group $\GU(n+1,q^2)$. If a soluble subgroup $G$ of $\GU(n+1,q^2)$ has order divisible by $\Phi^*_{2nd}(p)$, then $G$ fixes a subspace or quotient space $U$ of $V$ of dimension $n$ and  $G^U\leq \Gamma\textup{U}(1,q^{2n})$, where $G^U$ is obtained from the induced action of $G$ on $U$. Moreover, $G^U$ has order divisible by $\Phi^*_{2nd}(p)$.
\end{thm}
Theorem \ref{thm_BT} is applicable to the full preimage $G_0$ of $H\cap \PGU(n+1,q^2)$ in $\GU(n+1,q^2)$ by the preceding remarks. As a consequence, there exists a nonsingular $1$-dimensional subspace $P=\F_{q^2}\cdot v$  fixed by $H$ such that $U=P^\perp$ or $U=V/P$, cf. \cite[Propositions 4.1.4, 4.1.18]{KL129} or \cite[Table 2.3]{BHR407}. Since $\PGaU(n+1,q^2)$ acts transitively on the nonsingular points, by replacing $O$ with $g(O)$ for some $g\in\PGaU(n+1,q^2)$ if necessary, we can take $P$ to be a specific nonsingular $1$-dimensional subspace that we choose. Moreover, by \cite[Theorem 3.1]{BT2008Overgroup} the quotient image of the group  in Theorem \ref{thm_BT}  modulo  the scalars  is a subgroup of  the extension field type, and there is a unique conjugacy class of such subgroups in $\PGaU(n,q^2)$ by \cite[Proposition 4.3.6]{KL129}.

\section{The model of $H(n,q^2)$ and the Singer orbits}

In this section, we define and study the properties of the Singer orbits of $H(n,q^2)$, cf. \eqref{eqn_Sy} below. Also, we will obtain an algebraic characterization of the Singer-type ovoids when $q$ is even. Since there is no ovoid in $H(n,q^2)$ if $n>2$ is even, we assume that $n$ is odd in the sequel. We will establish a connection between a putative transitive ovoid and the Singer orbits in Lemma \ref{lem_Sy}, so the properties of the Singer orbits that we derive in this section, most notably Lemma \ref{lem_singer}, will play an important role in the proof of Theorem \ref{main}.

Set $V:=\F_{q^2}\times \F_{q^{2n}}$, and view it as an $(n+1)$-dimensional vector space over $\F_{q^2}$. We equip it with the following  Hermitian form:
\begin{equation}\label{eqn_Hf}
h((a,x))=a^{q+1}-\tr_{\F_{q^{2n}}/\F_{q^2}}(x^{q^n+1}), \quad\textup{for } (a,x)\in V.
\end{equation}
Let $H(n,q^2)$ be the associated Hermitian polar space. We define
\begin{equation}\label{eqn_rhophi}
\begin{array}{rcl}
\rho:&(a,x)\,\mapsto\, (a,\omega x),\\
\phi:&(a,x)\,\mapsto\, (a^p,x^p),
\end{array}
\end{equation}
where $\omega$ is an element of order $q^n+1$. We understand that both $\rho$ and $\phi$ act from the left, i.e., the image of a vector $v$ under $\rho$ (resp. $\phi$) is written as $\rho(v)$ (resp. $\phi(v)$). Both  of them lie in $\PGaU(n+1,q^2)$ and stabilize the point $P=\la(1,0)\ra$.  Let $G$ be the group generated by $\rho$ and $\phi$, which has order $2nd(q^n+1)$.
Set $U:=P^\perp$. Then $G$ also stabilizes $U$.

\begin{example}\label{rem_Singer}
Take $n=3$ and suppose that $q$ is even. The orbit $O$ of $(1,1)$ under the group $G=\la \rho,\phi\ra$ is an ovoid of $H(3,q^2)$, and it has $G$ as its full stabilizer when $q>2$.  In the case $q=2$, the full stabilizer of $O$ has order $324$. Any ovoid of $H(3,q^2)$ projectively equivalent to $O$ is referred to as a \textit{Singer-type} ovoid. This example is due to \cite{Cossidente2003Tovoid}, and its full stabilizer is determined in the same paper. The fact that a Singer-type ovoid has such an algebraic description follows from Lemma \ref{lem_S1} below.
\end{example}

\begin{example}\label{rem_q8}
Take $q=8$ and $n=3$, and let $\gamma$ be a primitive element of $\F_{2^9}$ with minimal polynomial $X^9+X^4+1$ over $\F_2$. By an exhaustive search using Magma \cite{Magma}, there are exactly two transitive ovoids with a soluble stabilizer in $\PGaU(4,8^2)$ up to projective equivalence:
\begin{enumerate}
\item[(1)] The orbit $O_1$ of $\la(1,\gamma^{39}) \ra$ under the group $H_1=\la \rho^9,\rho^3\phi^2\ra$ is a transitive ovoid,
\item[(2)] The orbit $O_2$ of $\la(1,\gamma^{109}) \ra$ under the group $H_2=\la\rho^9,\phi\ra$ is a transitive ovoid.
\end{enumerate}
In both cases, $H_i$ is the full stabilizer of $O_i$ in $\PGaU(4,8^2)$, $i=1,2$.
\end{example}

We define the following group homomorphism:
\begin{equation}\label{eqn_etahom}
  \eta:\, G \rightarrow \Aut(\F_{q^{2n}}),\;\rho^j\phi^i \mapsto \phi^i.
\end{equation}
Then $\ker(\eta)=\la \rho\ra$, and $\eta$ is surjective. In particular, $G$ has the structure $\Z_{q^n+1}:\,2nd$. We summarize some basic properties of $G$ in the following lemma.
\begin{lemma}\label{lem_Gpro}
	Let $G$ be as defined above. Then the following properties hold.
	\begin{enumerate}
		\item[(1)] $\phi\rho\phi^{-1}=\rho^p$, and for nonnegative integers $l,k,i$ it holds that
		\[
		(\rho^l\phi^k)^i=\rho^{(p^{ki}-1)l/(p^k-1)}\phi^{ki}.
		\]
		\item[(2)] If $g\in G\setminus \la\rho\ra$ has order $2$, then $g=\rho^j\phi^{nd}$ for some integer $j$.
	\end{enumerate}
\end{lemma}
\begin{proof}
We omit the proof of (1) which is routine. Suppose that $g=\rho^l\phi^k$ has order $2$, where $1\le k\le 2nd-1$. Then by (1) we have $2nd|2k$, $q^n+1|(1+p^k)l$, from which we deduce that $k=nd$. This proves (2).
\end{proof}

\medskip

For a point $\la(1,y)\ra$ of the polar space $H(n,q^2)$,  the  \textit{Singer orbit} $S_y$ of $\la(1,y)\ra$ is its image under the action of $\la \rho\ra$, i.e.,
\begin{equation}\label{eqn_Sy}
S_y:=\{\la(1,\omega^i y)\ra:\,0\leq i\leq q^n \}.
\end{equation}
Let $\omega_0$ be an element of order $(q^n+1)(q-1)$ such that  $\omega=\omega_0^{q-1}$. By Lemma \ref{lem_xomega}, a Singer orbit $S_y$  contains either a point $\la(1,x)\ra$ or a point $\la(1,x\omega_0)\ra$ for some $x\in\F_{q^n}^*$.  For each integer $i$ with $0\le i\le \frac{q^n+1}{q+1}-1$, we define
\begin{equation}\label{eqn_Liy}
L_{i,y}:=\{\la(1,\,\omega^{i+j({q^n+1})/({q+1})}y)\ra:\, 0\leq j\leq q\}.
\end{equation}
Each $L_{i,y}$ has size $q+1$, and they form a partition of $S_y$.  The subscript $i$ is understood to be taken modulo $\frac{q^n+1}{q+1}$, so that $L_{i,y}$ is defined for each integer $i$.

\begin{lemma}\label{lem_omegai}
Let $q$ be even and $n$ be an odd integer with $n\geq 5$. Then there exists an element $z\in\F_{q^{2n}}$ such that $z^{q^n+1}=1$, $\tr_{\F_{q^{2n}}/\F_{q^2}}(z)=1$ and $z\ne 1$.
\end{lemma}
\begin{proof}
Take $\delta$ to be an element of $\F_{q^2}$ such that $\delta+\delta^q=1$, and set $v:=\delta^{q+1}\in\F_q^*$. Then $1,\,\delta$ form a basis of $\F_{q^2}$ over $\F_q$. Since $n$ is odd, they also form a basis of $\F_{q^{2n}}$ over $\F_{q^n}$. Moreover, the minimal polynomial of $\delta$ over $\F_q$ is $X^2+X+v$, which is irreducible over $\F_q$.

Let $N$ be the number of elements $z\in \F_{q^{2n}}$ such that $\tr_{\F_{q^{2n}}/\F_{q^2}}(z)=1$, $z^{q^n+1}=1$. There exist $x,y\in\F_{q^n}$ such that $z=1+x+y\delta$. Upon expansion, the first equation reduces to $\tr_{\F_{q^n}/\F_q}(x)=0$, $\tr_{\F_{q^n}/\F_q}(y)=0$, and the second reduces to  $x^2+y^2v+xy+y=0$.

Let $\psi$ be the canonical additive character of $\F_{q^n}$, i.e., $\psi(x)=(-1)^{\tr_{\F_{q^n}/\F_2}(x)}$ for $x\in\F_{q^n}$. In particular, we have $\psi(x^2)=\psi(x)$ for each $x$. For a divisor $d$ of $n$, it is well known that $\sum_{x\in\F_{q^d}}\psi(ax)=q^d$ if $\tr_{\F_{q^n}/\F_{q^d}}(a)=0$,  and $=0$ otherwise, cf. \cite{LidlFF}.  We thus  have
\begin{equation*}
\begin{array}{rl}
	q^{n+2}N=&\sum_{x,y\in\F_{q^n}}\sum_{a,b\in\F_q}\sum_{z\in\F_{q^n}}\psi( ax+by+z(x^2+y^2v+xy+y))\\
	=&q^{2n}+\sum_{a,b\in\F_q}\sum_{x,y\in\F_{q^n}}\sum_{z\in\F_{q^n}^*}\psi( ax+by+z(x^2+y^2v+xy+y))\\
	=&q^{2n}+\sum_{a,b\in\F_q}\sum_{x,y\in\F_{q^n}}\sum_{z\in\F_{q^n}^*}\psi( ax+by+xz^{1/2}+y(zv)^{1/2}+zxy+zy)\\ =&q^{2n}+\sum_{a,b\in\F_q}\sum_{y\in\F_{q^n}}\sum_{z\in\F_{q^n}^*}\psi(by+y(zv)^{\frac{1}{2}}+yz)
\cdot\sum_{x\in\F_{q^n}}\psi(x(a+z^{1/2}+yz)).
\end{array}
\end{equation*}
The last sum is nonzero only if $a+z^{1/2}+yz=0$, i.e., $y=az^{-1}+z^{-1/2}$. We thus have
\begin{equation*}
	\begin{array}{rl} q^{n+2}N=&q^{2n}+q^n\sum_{a,b\in\F_q}\sum_{z\in\F_{q^n}^*}\psi(baz^{-1}+bz^{-1/2}+az^{-1/2}v^{1/2}+v^{1/2}+a+z^{1/2})\\
	=&q^{2n}+q^n\sum_{a,b\in\F_q}\psi(v+a) K(\psi;1,b^2+ ba+a^2v).
	\end{array}
	\end{equation*}
Here, $K(\psi;1,u)=\sum_{z\in\F_{q^n}^*}\psi(z+u/z)$ is a Kloosterman sum for $u\ne 0$, and $K(\psi;1,u)=-1$ for $u=0$. In both cases, we have $|K(\psi;1,u)|\leq 2q^{n/2}$  by \cite[Theorem 5.45]{LidlFF}. Since $X^2+X+v$ is irreducible over $\F_q$, we deduce that the only pair $(a,b)$ in $\F_q\times\F_q$ such that $b^2+ba+a^2v=0$ is $(0,0)$. It follows that
\begin{align*}
q^{n+2}N&\ge q^{2n}+q^n\psi(v)-(q^2-1)\cdot 2q^{3n/2}\ge q^{2n}-q^n-(q^2-1)\cdot 2q^{3n/2}.
\end{align*}
That is, $N\ge q^{n-2}-q^{-2}-2(q^2-1)q^{n/2-2}$.  We deduce that $N>1$ unless $(q,n)=(2,5)$, and this latter case can be verified directly.  This  completes the proof.
\end{proof}
\medskip

\begin{lemma}\label{lem_S1}
Suppose that $q$ is even and $n$ is an odd integer, $n\ge 3$. Then the Singer orbit $S_1$ is an ovoid of $H(n,q^2)$ if and only if $n=3$.
\end{lemma}
\begin{proof}
The set $S_1$ is an ovoid if and only if no two points of $S_1$ are perpendicular. This amounts to $1+\tr_{\F_{q^{2n}}/\F_{q^2}}(\omega^{k})\ne0$ for $1\le k\le q^n$, where $o(\omega)=q^n+1$. In the case $\omega^k\in\F_{q^2}$, this condition clearly holds. So we only need to consider the case $\omega^k\not\in\F_{q^2}$. The claim now follows from Lemma \ref{lem_wiproperty} for $n=3$ and Lemma \ref{lem_omegai} for $n\ge 5$.
\end{proof}
\medskip

To better understand the structure of the Singer orbits, we need to consider their interplay with the following subset of $H(n,\,q^2)$:
\begin{equation}\label{eqn_T}
T:=\{\la(0,t)\ra:\,t\in\F_{q^n}^*\;\textup{such that}\;  \tr_{\F_{q^n}/\F_q}(t^2)=0 \},
\end{equation}
which contains  $\frac{q^{n-1}-1}{q-1}$ distinct points in $H(n,q^2)$. Let $\Pi:=\{\la (0,x)\ra:\,x\in\F_{q^n}^*\}$ be a Baer subgeometry of the hyperplane $P^\perp$, where $P=\la (1,0)\ra$. In the case $q$ is even, these points form a hyperplane of $\Pi$. In the case $q$ is odd, they form a parabolic quadric on $\Pi$.
\begin{lemma}\label{lem_singer}
Take notation as above. Choose $R_t=\la(0,t)\ra\in T$ and $x\in\F_{q^n}^*$, and set $y=x$ or $y=x\omega_0$, where $o(\omega_0)=(q^n+1)(q-1)$. If $U:=R_t^{\perp} \cap S_y$, then  $|U|= k(q+1)$ for some integer $k$.  Moreover, if $k$ is odd, then $\tr_{\F_{q^n}/\F_{q}}(xt)=0$ in both cases.
\end{lemma}
\begin{proof}
A point $\la(1,\omega^iy)\ra$ lies in $R_t^\perp$ if and only if
\begin{equation}\label{eqn_wiyt}
  \tr_{\F_{q^{2n}}/\F_{q^2}}(w^iyt)=0.
\end{equation}
The following are two easy consequences.
\begin{enumerate}
	\item[(a)] The element $\omega^{(q^n+1)/(q+1)}$ lies in $\F_{q^2}$, since $\omega$ has order $q^n+1$. Therefore, the point $\la(1,\omega^iy)\ra$ lies in $R_t^\perp$ implies that $L_{i,y}\subseteq R_t^\perp$. That is, $U$ is the union of some $L_{i,y}$'s. As a corollary, $|U|=k(q+1)$ for some integer $k$.
	\item[(b)] Take $a$ to be the unique integer such that $0\le a\le q^n$ and $\omega^a=y^{q^n-1}$. By raising both sides of \eqref{eqn_wiyt} to the $q^n$-th power, we get $\tr_{\F_{q^{2n}}/\F_{q^2}}(\omega^{a-i}yt)=0$. That is, $L_{i,y}\subseteq U$ implies that $L_{a-i,y}\subseteq U$.
\end{enumerate}

Write $M:=\frac{q^n+1}{q+1}$, which is an odd integer. The two subsets $L_{i,y}$ and $L_{a-i,y}$ are equal if and only if $i\equiv a-i\pmod{M}$, i.e., $2i\equiv a\pmod{M}.$
Since $M$ is odd, there is exactly one integer $i_0$ such that $0\le i_0\le M-1$ and $L_{i_0,y}=L_{a-i_0,y}$. To be specific, if $y=x$, then $a=0$ and $i_0=0$; if $y=x\omega_0$, then $a=\frac{q^n-1}{q-1}$ and $i_0\equiv\frac{q^n-1}{q-1}t\pmod{M}$, where $t$ is the inverse of $2$ modulo $M$. In the later case,  we have $(q-1)i_0+1=(q^n-1)t+2t \equiv 0\pmod M$, i.e., $\omega^{i_0}\omega_0 \in \F_{q^2}^*$.

If $k$ is odd, then we must have $L_{i_0,y}\subseteq U$ by (b). That is, \eqref{eqn_wiyt} holds with $i=i_0$. In both cases, we plug in the expressions of $i_0$ and $y$ and deduce that $\tr_{\F_{q^{2n}}/\F_{q^2}}(xt)=0$ upon simplification. The second claim now follows, since $xt$ is in $\F_{q^n}$.
\end{proof}
\medskip

\begin{thm}\label{thm_singerovoid}
Let $n\ge 3$ be an odd integer. For a point $\la(1,y)\ra$ of $H(n,q^2)$, define $S_y$ as in \eqref{eqn_Sy}. Then  $S_y$ is not an ovoid of $H(n,q^2)$ unless $S_y$ is the Singer-type ovoid $S_1$ of $H(3,q^2)$ for even $q$.
\end{thm}
\begin{proof}
By Lemma \ref{lem_xomega}, $S_y$ contains an element $\la (1,x)\ra$ or $\la (1,x\omega_0)\ra$, where $x\in\F_{q^n}^*$ and $o(\omega_0)=(q^n+1)(q-1)$. We assume without loss of generality that $y=x$ or $y=x\omega_0$ for some $x\in\F_{q^n}^*$. Moreover, if $q$ is even we only need to consider the case $y=x\in\F_{q^n}^*$.

Suppose to the contrary that $S_y$ is an ovoid of $H(n,q^2)$. Let $T$ be as in \eqref{eqn_T}. By \eqref{eqn_def}, we have $|R_t^\perp\cap S_y|=q^{n-2}+1$, where $R_t=\la (0,t)\ra\in T$ for some $t\in\F_{q^n}$. Since $n$ is odd,  $\frac{q^{n-2}+1}{q+1}$ is odd.   By Lemma \ref{lem_singer}, we have $\tr_{\F_{q^n}/\F_{q}}(xt)=0$. This holds for all $t\in\F_{q^n}^*$ such that $\tr_{\F_{q^n}/\F_{q}}(t^2)=0$.
\begin{enumerate}
\item[(1)]In the case $q$ is even, it follows that $y=x\in\F_{q}^*$. Since $\la(1,y)\ra$ lies on $H(n,q^2)$, we deduce that $y=1$. The claim now follows from Lemma \ref{lem_S1}.
\item[(2)]In the case $q$ is odd, this implies that the points of $T$ lie on the hyperplane $\{\la (0,t)\ra:\,\tr_{\F_{q^n}/\F_{q}}(xt)=0\}$ of the Baer subgeometry $\Pi=\PG(\{0\}\times\F_{q^n})$ of $\PG(\{0\}\times\F_{q^{2n}})$. This is impossible, since $T$ is a nondegenerate parabolic quadric on $\Pi$.
\end{enumerate}
This completes the proof.
\end{proof}
\medskip

\section{The classification of transitive ovoids of $H(n,q^2)$}

In this section, we give the proof of Theorem \ref{main}. By \cite{Thas1981} there is no ovoid in $H(n,q^2)$ when $n$ is even, so we assume that $n$ is odd with $n\ge 3$. Let $q=p^d$ with $p$ prime. Suppose that $O$ is a transitive ovoid of $H(n,q^2)$ whose stabilizer in $\PGaU(n+1,q^2)$ is $H$. Since the case $H$ is insoluble has been settled in \cite{Bamberg2009Classificationovoid}, we only need to consider the case where $H$ is soluble. By Theorem  \ref{thm_BT} and the arguments following it, we can take the model of $H(n,q^2)$ as introduced in Section 3, and assume without loss of generality that $H$ is a subgroup of $G=\la\rho,\,\phi\ra$, where $\rho$ and $\phi$ are as in \eqref{eqn_rhophi}.

Let $\omega_0$ be a fixed element of order $(q^n+1)(q-1)$ in $\F_{q^{2n}}$, and write $\omega=\omega_0^{q-1}$. We keep the model of $H(n,q^2)$ and the notation as introduced in Section 3. Let $S_y$ and $T$ be as in \eqref{eqn_Sy} and \eqref{eqn_T}.
By Lemma \ref{lem_xomega}, there exists some $g\in G$ such that $g(O)$ contains a point of one of  the following forms:
\begin{enumerate}
\item[(1)] $\la(1,\,x)\ra$ with $x\in\F_{q^n}^*$;
\item[(2)] $\la(1,\,x\omega_0)\ra$ with $x\in\F_{q^n}^*$,  in the case $q$ is odd.
\end{enumerate}
If $q$ is even, only the first form occurs. Therefore, we assume without loss of generality that $O$ contains a point $\la (1,y)\ra$ of one of the two prescribed forms.

\subsection{Restrictions on the parameters}\label{subsec_res}

Since $H$ acts transitively on $O$, there is a positive integer $m$ such that $|H|=m(q^n+1)$.  Write $H\cap\la \rho\ra=\la\rho^s\ra$, where $s|q^n+1$. In the case $s=1$, $O=S_y$ for some $y\in\F_{q^{2n}}^*$, which has been settled in Theorem \ref{thm_singerovoid}, so we assume that $s>1$ below. By considering the restriction of the homomorphism $\eta:\,G\rightarrow\Aut(\F_{q^{2n}})$ to $H$, we see that
\begin{equation}\label{eqn_gpH}
H=\la\rho^s,\,\rho^j\phi^k\ra
\end{equation}
for some integers $k,j$ such that $0\le j\le s-1$,  $k|2nd$ and $\frac{q^n+1}{s}\cdot\frac{2nd}{k}=|H|$. The last equality is equivalent to $2nd=mks$.

\begin{lemma}\label{lemma_Hform}
We have $mks=2nd$, $s|(q^n+1)$, and $\frac{p^{2nd}-1}{p^k-1}\cdot j\equiv 0\pmod{s}$.
\end{lemma}
\begin{proof}
By (1) of Lemma \ref{lem_Gpro}, $(\rho^j\phi^k)^{2nd/k}=\rho^{j(p^{2nd}-1)/(p^k-1)}$, so it lies in $H\cap\la\rho\ra=\la\rho^s\ra$. This gives this last congruence in the lemma.
\end{proof}
\medskip

\begin{corollary}\label{cor_s}
	If $q$ is even, then $s$ is odd; if $q$ is odd, then $s$ is not a multiple of $4$.
\end{corollary}
\begin{proof}
	The case $q$ is even is clear, since $q^n+1$ is odd. Assume that $q$ is odd, and suppose to the contrary that $s$ is multiple of $4$. By $2nd=mks$, we deduce that $2|d$. Then $q^n+1=p^{nd}+1\equiv 2\pmod{4}$, so $q^n+1$ is not divisible by $4$. This contradicts the fact that $s$ divides $q^n+1$.
\end{proof}
\medskip

\begin{lemma}\label{lem_case2cond}
Suppose that $q=p^d$ is odd and $s$ is even, and let $\la(1,y)\ra$ be a point of $O$.
\begin{enumerate}
\item[(1)] If $m$ is even and $y\in \F_{q^n}^*$, then $y\in \F_{p^{ks}}^*$.
\item[(2)] If $y=x\omega_0$ for some $x\in \F_{q^n}^*$, then
		$m$ is odd and $y^{(p^{ks/2}-1)(q^n+1)}=1$.
\end{enumerate}
\end{lemma}
\begin{proof}
	Recall that $H$ is the stabilizer of the ovoid $O$ in the group $G$, and $\omega_0$ is an element of order $(q^n+1)(q-1)$ such that $\omega=\omega_0^{q-1}$. Since $H$ has order $m(q^n+1)$ and acts transitively on $O$, the stabilizer $A$ of $\la(1,y)\ra$ in $H$ has order $m$. It is clear that $A\cap\la \rho\ra=1$, so $\eta(A)$ has order $m$, where $\eta$ is the homomorphism in \eqref{eqn_etahom}. It follows that $A=\la \rho^l\phi^{ks}\ra$ for some nonnegative integer $l$, where we used the fact $2nd=mks$ in Lemma \ref{lemma_Hform}. Since
	\[
	\rho^l\phi^{ks}(\la(1,y)\ra)=\la(1,y^{p^{ks}}\omega^l)\ra=\la (1,y)\ra,
	\]
	we have
	\begin{equation}\label{eqn_yomega}
	y^{p^{ks}-1}=\omega^{-l}.
	\end{equation}
	Since $H=\la\rho^s,\rho^j\phi^k\ra$ and $\eta(\rho^l\phi^{ks})=\eta(\rho^j\phi^k)^s$, there exists an  integer $a$ such that
	\[
	\rho^{sa}(\rho^j\phi^{k})^s=\rho^{sa+j(p^{ks}-1)/(p^k-1)}\phi^{ks}=\rho^l\phi^{ks}\in A.
	\]
Here we have used (1) of Lemma \ref{lem_Gpro} in the first equality. It follows that
	\begin{equation}\label{eqn_samod1}
	l\equiv sa+j\frac{p^{ks}-1}{p^k-1} \pmod{q^n+1}.
	\end{equation}
Since $q$ is odd and $s$ is even, we deduce that  $\frac{p^{ks}-1}{p^k-1}=\frac{p^{ks}-1}{p^{2k}-1}\cdot(p^k+1)$ is even. It follows that $l$ is even. We are now ready to prove the two claims.
	
(1). Assume that $m$ is even and $y\in\F_{q^n}^*$. In this case,  $\omega^{l}=y^{1-p^{ks}}$ lies in both $\F_{q^n}^*$ and $\la w\ra$, so has order dividing $\gcd(q^n+1,q^n-1)=2$.  On the other hand, $d$ is even since $2nd=mks$ and both $m,s$ are even. It follows that $q^n+1=p^{nd}+1\equiv 2\pmod{4}$. Since $l$ is even, $\omega^l$ has odd order. We conclude that $\omega^l=1$ and thus $y^{p^{ks}-1}=1$. This proves the first claim.

(2). Assume that $y=x\omega_0$ for some $x\in \F_{q^n}^*$. Recall that $\omega_0^{q-1}=\omega$. The element  $\rho^{-(q^n-1)/(q-1)}\phi^{nd}$ stabilizes $\la(1,\,x\omega_0)\ra$, since
	\[	 \rho^{-(q^n-1)/(q-1)}\phi^{nd}(\la(1,\,x\omega_0)\ra)=\la(1,\,x\omega_0^{q^n}w^{-(q^n-1)/(q-1)})\ra
	=\la(1,\,x\omega_0)\ra.
	\]
We first show that $m$ is odd. Suppose to the contrary that $m$ is even. We have $ks|nd$ from $mks=2nd$, so there is an element $\rho^v\varphi^{nd}$ in $A$. It follows that $\rho^{v+{(q^n-1)}/{(q-1)}}=(\rho^v\varphi^{nd})(\rho^{-(q^n-1)/(q-1)}\phi^{nd})^{-1}\in A$, yielding $v=-(q^n-1)/(q-1)$. Therefore, there exist integers $i,\,b$ such that
	\begin{equation}\label{eqn_1}
	\rho^{si}(\rho^{l}\phi^{ks})^{b} =\rho^{si+
		l(p^{ksb}-1)/(p^{ks}-1)}\phi^{ksb}=\rho^{-({q^n-1})/({q-1})}\phi^{nd}.
	\end{equation}
Here we have used (1) of Lemma \ref{lem_Gpro} in the first equality. We deduce that
\[
si+\frac{p^{ksb}-1}{p^{ks}-1}l\equiv -\frac{q^n-1}{q-1}\pmod{q^n+1}.
\]
The left hand side is even, since $l$ and $s$ are even. This leads to the contradiction that ${(q^n-1)}/{(q-1)}$ is even. We conclude that $m$ is odd.
		
Write $y=\gamma^{h}$, where $\gamma$ is  a primitive element of $\F_{q^{2n}}$. Since $l$ is even, we raise both sides of \eqref{eqn_yomega} to the $\frac{q^n+1}{2}$-th power and deduce that
\[
	h(p^{ks}-1)\cdot\frac{(q^n+1)}{2}\equiv 0 \pmod {q^{2n}-1},\,\textup{i.e.,}\,  h(p^{ks}-1)/2\equiv 0\pmod{q^n-1}.
\]
Since $m$ is odd and $ks=\frac{2nd}{m}$, we have $\gcd(ks,nd)=\frac{nd}{m}\cdot\gcd(2,m)=\frac{nd}{m}$. It follows that
	\[
	\gcd(p^{ks}-1,q^n-1)=p^{\gcd(ks,nd)}-1=p^{nd/m}-1.
	\]
Since $\frac{p^{ks}-1}{p^{nd/m}-1}=p^{nd/m}+1$ is even,  we see that $p^{nd/m}-1$ also divides $\frac{p^{ks}-1}{2}$. Hence  $p^{nd/m}-1$ divides $\gcd\left(\frac{p^{ks}-1}{2},q^n-1\right)$, which clearly divides $\gcd(p^{ks}-1,q^n-1)$. We thus have  $\gcd\left(\frac{p^{ks}-1}{2},q^n-1\right)=p^{nd/m}-1$. It follows that $h$ is a multiple of $\frac{q^n-1}{p^{nd/m}-1}$. Therefore, $y^{(p^{nd/m}-1)(q^n+1)}=1$ as desired. This completes the proof.
\end{proof}
\medskip

For each integer $i$, $\rho^i(O)$ is also an ovoid. Set
\begin{equation}\label{eqn_FO}
\mathfrak{O}:=\bigcup_{i=0}^{s-1} \rho^i(O),
\end{equation}
which is a multiset. Recall that the Singer orbit $S_y$ is defined in \eqref{eqn_Sy}.
The following lemma will play a crucial role in our arguments.
\begin{lemma}\label{lem_Sy}
	Take notation as above, and let $\la(1,y)\ra$ be a point of $O$.
	\begin{enumerate}
		\item[(i)]We have $\mathfrak{O}=\bigcup_{i=0}^{s-1} \phi^{ik}(S_y)$, and
		$\sum_{i=0}^{s-1}|R_t^{\perp}\cap  \phi^{ik}(S_y)|=s(q^{n-2}+1)$ for any point $R_t=\la (0,t)\ra$ in $T$, where $T$ is as  in \eqref{eqn_T};
		\item[(ii)]We have  $S_y=\phi^{ks}(S_y)$ and $y^{(p^{ks}-1)(q^n+1)}=1$.
	\end{enumerate}
\end{lemma}	
\begin{proof}
	Let $C_y$ be the image of $\la (1,y)\ra$ under $\la\rho^s\ra$; in particular, $\rho^s(C_y)=C_y$. Then $S_y=\bigcup_{i=0}^{s-1}\rho^i(C_y)$, and $O=\bigcup_{a=0}^{s-1}(\rho^j\phi^k)^a(C_y)$. We have
	\begin{align*}
	\mathfrak{O}&=\bigcup_{i,a=0}^{s-1} \rho^i(\rho^j\phi^k)^a(C_y)= \bigcup_{i,a=0}^{s-1} \rho^{i+j(p^{ka}-1)/(p^k-1)}\phi^{ka}(C_y)\\
	&=\bigcup_{i,a=0}^{s-1} \phi^{ka}\rho^{(i+j(p^{ka}-1)/(p^k-1))p^{2nd-ka}}(C_y)=\bigcup_{i=0}^{s-1} \phi^{ik}(S_y).
	\end{align*}
In the last equality, we have used the fact that
	\[
	\left\{\left(i+j\frac{p^{ka}-1}{p^k-1}\right)p^{2nd-ka}\pmod{s}:\, 0\leq i\leq s-1\right\}=\{0,1,\cdots,s-1\}
	\]
for a given integer $a$. Here $j,k$ are constants such that $H=\la\rho^s,\rho^j\phi^k\ra$. This proves the first equation in (i).  Since a point $R_t=\la (0,t)\ra\in T$ has a zero first coordinate, it does not lie in any of the $\rho^i(O)$'s, and thus $|R_t^\perp\cap \rho^i(O)|=q^{n-2}+1$ for each $i$ by \eqref{eqn_def}. The second equation in (i) now follows from $\mathfrak{O}=\bigcup_{i=0}^{s-1} \phi^{ik}(S_y)$.
	
	We now prove (ii). By induction we have $\phi^b\rho^a=\rho^{ap^b}\phi^b$ from  (1) of Lemma \ref{lem_Gpro}, where $a,b$ are nonnegative integers. Applying $\phi^k$ to $\mathfrak{O}$, we have
	\[
	\phi^k(\mathfrak{O})=\bigcup_{i=0}^{s-1} \phi^k\rho^{i}(O)=\bigcup_{i=0}^{s-1}\rho^{ip^k-j} (\rho^j\phi^k(O))=\bigcup_{i=0}^{s-1}\rho^{ip^k-j}(O).
	\]
	Here we have used the fact that $\rho^j\phi^k$ stabilizes $O$. The set  $\{ip^{k}-j\pmod{s}:\, 0\leq i\leq s-1\}$ equals $\{0,1,\cdots,s-1\}$ and $\rho^s$ stabilizes $O$, so $\phi^k(\mathfrak{O})=\mathfrak{O}$. It follows that $\phi^{ks}(S_y)=S_y$.
As a consequence, there is some integer $i$ such that
	\[
	\phi^{ks}(\la(1,y)\ra)=\la(1,\,y^{p^{ks}})\ra=\la(1,\,\omega^iy)\ra\in S_y,\;\textup{i.e.,}\; y^{p^{ks}}=\omega^iy.
	\]
It follows that $y^{(p^{ks}-1)(q^n+1)}=1$.  This completes the proof.
\end{proof}
\medskip

\begin{lemma}\label{lem_noovoideven}
If $q=2^d$ and $2^d\ge nd$, then $s\ge q+1$ unless $n=3$ and $O$ is a Singer-type ovoid.
\end{lemma}
\begin{proof}
By the arguments in the beginning of this section, we assume without loss of generality that $O$ contains a point $\la (1,y)\ra$ with $y\in\F_{q^n}^*$. Write $X:=\{y^{p^{ki}}: 0\leq i\leq s-1\}$.

We claim that $\tr_{\F_{q^n}/\F_q}(y^{p^{ki}}t)=0$ for some $i$ with $0\le i\le s-1$, where $t$ is any element of $\F_{q^n}^*$ with $\tr_{\F_{q^n}/\F_q}(t)=0$. Let $t$ be such an element, and set $R_t:=\la (0,t)\ra$ which is any element of $T$.  By (i) of Lemma \ref{lem_Sy}, $\sum_{i=0}^{s-1}|R_t^{\perp}\cap \phi^{ki}(S_y)|=s(q^{n-2}+1)$. Each term in the summation is a multiple of $q+1$ by Lemma \ref{lem_singer}. By Corollary \ref{cor_s}, $s$ is odd; also, it is easy to see that $\frac{q^{n-2}+1}{q+1}$ is odd. It follows that there is at least one $i$ such that $|R_t^{\perp}\cap \phi^{ki}(S_y)|=u(q+1)$ for some odd integer $u$. Since $\phi^{ki}(S_y)=S_{y^{p^{ki}}}$, the claim now  follows from Lemma \ref{lem_singer}.
	
In the case $y\in\F_q^*$, we deduce that $y=1$ by the fact $\la(1,y)\ra$ is a singular point. We have $\phi^{ki}(S_1)=S_1$ for each $i$ and thus $\mathfrak{O}=s\cdot S_1$, which implies that $O=S_1$. By Lemma \ref{lem_S1}, $S_1$ is an ovoid if and only if $n=3$.

It remains to check the case $y\not\in\F_q^*$. Assume to the contrary that $s\le q$. We then have $\lceil\frac{q^{n-1}-1}{s}\rceil\ge \lceil q^{n-2}-q^{-1}\rceil=q^{n-2}$. By the pigeonhole principle, there exists some $i$ such that $\tr_{\F_{q^n}/\F_q}(y^{p^{ki}}t)=0$ for at least $q^{n-2}$ \textit{nonzero} elements $t\in\F_{q^n}$ such that $\tr_{\F_{q^n}/\F_q}(t)=0$. We deduce that $y^{p^{ki}}\in\F_q^*$, i.e., $y\in\F_q^*$: a contradiction. Hence the case $y\not\in\F_q^*$ can not occur. This completes the proof.
\end{proof}
\medskip

We now consider the odd characteristic case. The proof follows the same line as in that of Lemma \ref{lem_noovoideven}. We need the following technical lemma.
\begin{lemma}\label{lem_int}
Suppose that $q=p^d$ is odd and  $\la(1,y)\ra$ is an element of $O$, where $x:=y\in\F_{q^n}^*$ or $x:=y\omega_0^{-1}\in\F_{q^n}^*$. Set
$X:=\{x^{p^{ka}}:\,0\le a\le s-1\}$; if $s$ is even, set $X_h:=\{x^{p^{ka}}:\,0\le a\le s/2-1\}$. Let $R_t=\la (0,t)\ra$ be an element of $T$, where $T$ is as in \eqref{eqn_T}.
\begin{enumerate}
\item[(i)] If $m$ or $s$ is odd, then there is an element of $X$ such that $\tr_{\F_{q^n}/\F_q}(xt)=0$;
\item[(ii)] If $m$ and $s$ are both even and $t\in\F_{p^{nd/e}}$, then there is an element of $X_h$  such that $\tr_{\F_{q^n}/\F_{q}}(xt)=0$, where $e$ is the highest power of $2$ dividing $m$.
\end{enumerate}
\end{lemma}
\begin{proof}
In the case $s$ is odd, the argument is exactly the same as in the proof of Lemma \ref{lem_noovoideven}, so we assume that $s$ is even below. By Corollary \ref{cor_s}, we have $s=2s_0$ for an odd integer $s_0$. By Lemma \ref{lemma_Hform}, we have $2nd=mks$. It follows that $nd=mks_0$. Also, we have we have $S_y=\phi^{ks}(S_y)$ and $y^{(p^{ks}-1)(q^n+1)}=1$ by (ii) of Lemma \ref{lem_Sy}.
	
First consider the case $m$ is odd. We claim that $S_y=\phi^{ks_0}(S_y)$. To prove this, we consider two separate cases.
\begin{enumerate}
\item[(i)] If $y\in\F_{q^n}^*$, then the order of $y$ divides $D:=\gcd(q^n-1,(q^n+1)(p^{ks}-1))$. We have
   \begin{align*}
    D&=\gcd(q^n-1,2(p^{ks}-1))=(p^{ks_0}-1)\cdot\gcd\left(\frac{q^n-1}{p^{ks_0}-1},2(p^{ks_0}+1)\right)\\ &=(p^{ks_0}-1)\cdot\gcd\left(\frac{q^n-1}{p^{ks_0}-1},p^{ks_0}+1\right)=\gcd(q^n-1,p^{2ks_0}-1)\\
    &=p^{\gcd(mks_0,2ks_0)}-1=p^{ks_0}-1.
	\end{align*}
    Here, we have used the fact that  $\frac{q^n-1}{p^{ks_0}-1}=1+p^{ks_0}+\cdots+p^{(m-1)ks_0}$ is odd in the third equality. It follows that $y\in\F_{p^{ks_0}}^*$, and so $S_y=\phi^{ks_0}(S_y)$.
\item[(ii)] If $y=x\omega_0$ for some $x\in\F_{q^n}^*$, then by Lemma \ref{lem_case2cond} we have $y^{(p^{ks_0}-1)(q^n+1)}=1$, i.e., $y^{p^{ks_0}}\in y\la\omega\ra$. In this case, we have $\phi^{ks_0}(S_y)=S_y$ as desired.
\end{enumerate}
In both cases, $\phi^{ki}(S_y)=\phi^{k(s_0+i)}(S_y)$ for each integer $i$. We thus have $\sum_{i=0}^{s_0-1}|R_t^{\perp } \cap \phi^{ki}(S_y)|=s_0(q+1)$ by (i) of Lemma \ref{lem_Sy}. The claim (i) now follows by a parity argument and invoking Lemma \ref{lem_singer} as in the proof of Lemma \ref{lem_noovoideven}.	
	
Next we consider the case $m$ is even. Let $e$ be the highest power of $2$ that divides $m$. From $nd=mks_0$ and the fact $n$ is odd we deduce that $e$ divides $d$. Set $q_1:=p^{d/e}$, and  define a subset $T_1$ as follows:
	\begin{equation}\label{eqn_T1}
	T_1:=\{\la(0,\,t)\ra:\, t \in\F_{q_1^n}^*\;\textup{such that}\;\tr_{\F_{q_1^n}/\F_{q_1}}(t^2)=0\}.
	\end{equation}
Since $\gcd(d,nd/e)=\frac{d}{e}\cdot\gcd(e,n)=d/e$, we have $\F_{q}\cap\F_{q_1^n}=\F_{q_1}$. It follows that $\tr_{\F_{q_1^n}/\F_{q_1}}(t^2)=\tr_{\F_{q^n/\F_q}}(t^2)$ for $t\in \F_{q_1^n}$. Hence $T_1$ is a subset of $T$ of size $\frac{q_1^{n-1}-1}{q_1-1}$.

By (2) of Lemma \ref{lem_case2cond}, the case $y\omega_0^{-1}\in\F_{q^n}^*$ does not occur, and by (1) of the same lemma we must have $y\in\F_{p^{ks}}^*$. For $R_t=\la (0,t)\ra$ in $T_1$, the element $\la(1,\,\omega^ay^{p^{ki}}\ra$ lies in $R_t^\perp\cap \phi^{ki}(S_y)$ if and only if
\begin{equation}\label{eqn_ayt}
\tr_{\F_{q^{2n}}/\F_{q^2}}(\omega^{a} y^{p^{ki}}t)=0.
\end{equation}
Observe that $q_1^n=p^{mks_0/e}$ and $m/e$ is odd. By raising both sides of \eqref{eqn_ayt} to the $p^{mks_0/e}$-th power, we deduce  from $y\in \F_{p^{ks}}$ that the condition \eqref{eqn_ayt} is equivalent to $\tr_{\F_{q^{2n}}/\F_{q^2}}(\omega^{ap^{mks_0/e}}y^{p^{ks_0+ki}}t)=0$.   It follows that  $\la (1,\omega^{a}y^{p^{k(s_0+i)}})\ra$ is in $R_t^\perp\cap \phi^{k(s_0+i)}(S_y)$. This gives a bijection between $R_t^\perp\cap \phi^{ki}(S_y)$ and $R_t^\perp\cap \phi^{k(s_0+i)}(S_y)$ via
\[
\la(1,\omega^{a}y^{p^{ki}})\ra\mapsto\la (1,\omega^{ap^{mks_0/e}}y^{p^{k(s_0+i)}})\ra.
\]
We thus have $\sum_{i=0}^{s_0-1}|R_t^{\perp } \cap \phi^{ki}(S_y)|=s_0(q^{n-2}+1)$ by (i) of Lemma \ref{lem_Sy}. The claim now follows as in the case $m$ is odd.
\end{proof}
\medskip

\begin{lemma}\label{lem_odd}
Suppose that $q=p^d$ is odd. Then we have:
\begin{enumerate}
\item[(i)] If one of $m,s$ is odd, then $s\ge\frac{q+1}{2}$ when $n=3$ and $s\ge q$ when $n\ge 5$;
\item[(ii)] If both $m$ and $s$ are even, then $s\ge p^{d/e}+1$ when $n=3$ and $s\ge 2p^{d/e}$ when $n\ge 5$, where $e$ is the highest power of $2$ dividing $m$.
\end{enumerate}
\end{lemma}
\begin{proof}
By the argument in the beginning of this section, we assume without loss of generality that $O$ contains a point $\la (1,y)\ra$ with $x:=y\in\F_{q^n}^*$ or $x:=y\omega_0^{-1}\in\F_{q^n}^*$. Let $X$ and $X_h$ be as in Lemma \ref{lem_int}, and write $x_i:=x^{p^{ki}}$.

First, consider the case where one of $m$ and $s$ is odd. Let $\Pi$ be the projective geometry with ambient $\F_q$-linear vector space $\F_{q^n}$, and $\cQ$ be the  quadric defined by the equation $Q(t)=\tr_{\F_{q^n}/\F_q}(t^2)$. It is easy to see that $\cQ$ is nondegenerate. Since $q$ and $n$ are odd, $\cQ$ is a parabolic quadric. Since $\la(1,y)\ra$ lies in $H(n,q^2)$, we deduce from the fact $x\in\F_{q^n}^*$ and the expression of the Hermitian form $h$ in \eqref{eqn_Hf} that $Q(x)=1$ or $Q(x)=\omega_0^{-(q^n+1)}$ according as $y=x$ or $y=x\omega_0$. Here, we recall that $o(\omega_0)=(q^n+1)(q-1)$. Let $\pi_i$ be the hyperplane of $\Pi$ with equation $\tr_{\F_{q^n}/\F_q}(x_it)=0$, $0\le i\le s-1$.  By Lemma \ref{lem_int},  each point $\la t\ra$ of $\cQ$ is on some hyperplane $\pi_i$. Since $Q(x_i)=Q(x)^{p^{ki}}\ne 0$, we see that each $\pi_i$ intersects $\cQ$ in a nondegenerate quadric, i.e., a $Q^+(n-2,q)$ or $Q^-(n-2,q)$. We thus have $|\cQ|\le |X|\cdot|Q^+(n-2,q)|$, i.e.,
\begin{align*}
|X|&\ge\left\lceil\frac{q^{n-1}-1}{(q^{(n-1)/2}-1)(q^{(n-1)/2-1}+1)}\right\rceil=\left\lceil\frac{q^{(n-1)/2}+1}{q^{(n-1)/2-1}+1}\right\rceil\\
&=q-\left\lfloor\frac{q-1}{q^{(n-1)/2-1}+1}\right\rfloor.
\end{align*}
When $n=3$, we get $|X|\ge\frac{q+1}{2}$; when $n\ge 5$, we get $|X|\ge q$. Since $|X|\le s$, the claim follows in this case.

Second, consider the case where both $m$ and $s$ are even. By Lemma \ref{lem_case2cond}, we have $x=y\in\F_{p^{ks}}^*$.  Let $e$ be the highest power of $2$ dividing $m$; in particular, we have $\gcd(e,n)=1$ since $n$ is odd. Recall that we have $2nd=mks$ by Lemma \ref{lemma_Hform}, so $e$ divides $d$.  Set $q_1:=p^{d/e}$, so that $q=q_1^e$. Choose a basis $\zeta_1,\cdots,\zeta_e$ of $\F_q$ over $\F_{q_1}$; then they also form a basis of $\F_{q^n}$ over $\F_{q_1^n}$ by Lemma \ref{lem_ind}. For each $i$, write $x_i=\sum_{j=1}^{e}x_{ij}\zeta_j$ with $x_{ij}\in\F_{q_1^n}$, and set $z_i$ to be some nonzero $x_{ij}$. Let $\Pi$ be the projective geometry with ambient $\F_{q_1}$-linear vector space $\F_{q_1^n}$, and $\cQ_1$ be the  quadric defined by the equation $Q_1(t)=\tr_{\F_{q_1^n}/\F_{q_1}}(t^2)$.
  It is easy to check that  $\cQ_1$ is nondegenerate. Since both $q_1$ and $n$ are odd,  $\cQ_1$ is a parabolic quadric. Since $\la(1,x_i)\ra$ lies in $H(n,q^2)$, we deduce as in the previous case that $\tr_{\F_{q^n/\F_q}}(x_i^2)=1$.  Let $\pi_i$ be the hyperplane of $\Pi$ with equation $\tr_{\F_{q_1^n}/\F_{q_1}}(z_it)=0$, $0\le i\le s/2-1$.  By Lemma \ref{lem_int}, for each point $\la t\ra$ of $\cQ_1$ there exists $i$ such that
 \[\tr_{\F_{q^n}/\F_q}(x_it)=0,\, \text{i.e.},\,
\sum_{j=1}^e\tr_{\F_{q_1^n}/\F_{q_1}}(x_{ij}t)\zeta_j=0.\]
It follows that $\tr_{\F_{q_1^n}/\F_{q_1}}(x_{ij}t)=0$ for each $j$, and so $\la t\ra$ is on the hyperplane $\pi_i$ by our choice of $z_i$. As in the previous case, $\pi_i$ intersect the quadric $\cQ_1$ is a $Q^-(n-2,q_1)$ or $Q^+(n-2,q_1)$.  It follows that $|\cQ_1|\le |Q^+(n-2,q_1)|\cdot s/2$. As in the previous case, we deduce that  $\frac{s}{2}\ge\frac{q_1+1}{2}$ when $n=3$, and  $\frac{s}{2}\ge q_1$ when $n\ge 5$. This completes the proof.
\end{proof}
\medskip

\subsection{The proof of Theorem \ref{main}}

We continue to use the notation that we have introduced in the beginning of this section. In particular, $q=p^d$ with $p$ prime, $O$ is a transitive ovoid of $H(n,q^2)$ that contains an element $\la(1,y)\ra$ with $x:=y\in\F_{q^n}^*$ or $x:=y\omega_0^{-1}\in\F_{q^n}^*$, and $H=\la \rho^s,\phi^j\rho^k\ra$ is the  stabilizer of $O$ in $\PGaU(n+1,q^2)$, where $\rho$ and $\phi$ are as defined in \eqref{eqn_rhophi}. We write $|H|=m(q^n+1)$. By Lemma \ref{lemma_Hform}, we have $mks=2nd$, $s|q^n+1$.

\begin{lemma}\label{lem_bound}
	Let $p$ be a prime, $d$ be a positive integer, and write $q=p^d$. Suppose that  $s$ is a divisor of $\gcd(6d,\, q^3+1)$.
	\begin{enumerate}
		\item[(1)] If $p=2$ and $d\geq 4$, then $s<q+1$.
		\item[(2)] If $p$ is odd, then $2s<q+1$ unless $(p,d)\in\{(3,1),\,(5,1),\, (11,1)\}$.
	\end{enumerate}
\end{lemma}
\begin{proof}
	In the case $p=2$, we have $\gcd(6d,\, q^3+1)=\gcd(3d,q^3+1)\le 3d$, so $s\le 3d<2^d+1$ for $d\ge 4$. In the case $p$ is odd, we consider two cases:
	\begin{enumerate}
		\item[(i)] If $p\equiv 0,\,1\pmod{3}$, then $q^3+1$ is not divisible by $3$, and so $\gcd(6d,q^3+1)=\gcd(2d,q^3+1)$. It follows that $2s\le 4d<p^d+1$ unless $(p,d)=(3,1)$ in this case.
		\item[(ii)] If $p\equiv 2\pmod{3}$, then $2s\le 12d<5^d+1\le p^d+1$ if $d\ge 2$. It remains to consider the case $d=1$. In this case, $2s\le 12<p+1$ if $p>11$.
	\end{enumerate}
	This completes the proof.
\end{proof}
\medskip

We first restate and prove the soluble case with $n=3$ in Theorem \ref{main}.
\begin{thm}[Theorem \ref{main}, soluble case with $n=3$]\label{thm_main2}
Let $p$ be a prime and let $q=p^d$. Suppose that $O$ is a  transitive ovoid of $H(3,q^2)$ whose stabilizer $H$ in $\PGaU(4,q^2)$ is soluble. Then $q$ is even, and $O$ is projectively equivalent to a Singer-type ovoid  or  one of the two  ovoids of $H(3,8^2)$ described in Example \ref{rem_q8}.
\end{thm}
\begin{proof}
As we remarked in the beginning of Subsection \ref{subsec_res}, the case $s=1$ has been settled in Theorem \ref{thm_singerovoid}, so we assume that $s>1$. By Lemma \ref{lemma_Hform} and Corollary \ref{cor_s}, we have $mks=6d$, $s|q^n+1$, and $s$ is odd if $q=p^d$ is even.

First consider the case $q$ is even. The case $d=1$ follows from the classification of Brouwer and Wilbrink \cite{Brouwer1990}. The case $d=2$ can be excluded, since $\gcd(6d,2^{3d}+1)=1$ in this case. In the case $d=3$, $\gcd(6d,2^{3d}+1)=9$, so $s=3$ or $s=9$. We enumerate the triples $(m,k,s)$ such that $mks=6d$, $s\in\{3,9\}$, and check whether the orbits of the subgroups of the form as in \eqref{eqn_gpH} form ovoids by Magma \cite{Magma}. It turns out that up to projective equivalence there are exactly two transitive ovoids besides the classical ovoids and the Singer-type ovoids, namely those listed in Example \ref{rem_q8}. In the case $d \geq 4$, we have $s< q+1$ by Lemma \ref{lem_bound}. Since $2^d\ge 3d$ in this case, we see that $O$ must be a Singer-type ovoid by Lemma \ref{lem_noovoideven}. 	

Next consider the case $q=p^d$ is odd. For $q\in \{3,5,11\}$, an exhaustive search using Magma \cite{Magma} shows that there is no ovoid $O$ of the prescribed form. The search strategy is the same as in the case $q=2^3$. So we assume that  $q\not\in\{3,5,11\}$ below. We have $2s<q+1$ by Lemma \ref{lem_bound}. If one of $m,s$ is odd,   this contradicts (i) of Lemma \ref{lem_odd}. If both $m$ and $s$ are even, by (ii) of Lemma \ref{lem_odd} we have $s\ge p^{d/e}+1$, where $e$ is the highest power of $2$ dividing $m$. On the other hand, $\gcd(3,q^3+1)=1$ since $d$ is even, so $\gcd(3,s)=1$. It follows from $mks=6d$ that $s$ divides $2d/e$. So we deduce that $2d/e\ge p^{d/e}+1$, which is impossible. This completes the proof.
\end{proof}
\medskip

We next consider the soluble case with $n\ge 5$ in Theorem \ref{main}.
\begin{thm}[Theorem \ref{main}, soluble case with $n\ge 5$]\label{thm_highdim}
Let $n$ be an odd integer with $n\ge 5$ and $q$ be any prime power. Then there is no transitive ovoid in $H(n,q^2)$ with a soluble stabilizer in $\PGaU(n+1,q^2)$.
\end{thm}
\begin{proof}
We first consider the case where $p>45$. If $n\ge\frac{p+1}{2}$, then there is no ovoid in $H(n,q^2)$ by Lemma \ref{lem_Fprop} and Theorem \ref{thm_moo}. If $n\le\frac{p-1}{2}$, then we consider two cases:
\begin{enumerate}
\item[(1)] If one of $m,s$ is odd, then we have $s\ge q=p^d$ by (i) of Lemma \ref{lem_odd}. On the other hand, $s|2nd$, so $s\le (p-1)d<pd<p^d$: a contradiction.
\item[(2)] If both $m$ and $s$ are even, then $s\ge 2p^{d/e}$ by (ii) of Lemma \ref{lem_odd}, where $e$ is the highest power of $2$ dividing $m$. On the other hand, from $mks=2nd$ we deduce that $s$ divides $2nd/e$, so $s\le (p-1)d/e<pd/e<2p^{d/e}$: a contradiction.
\end{enumerate}
This established the claim for the case $p>45$.

It remains to consider the case $p<45$. By (1) of Lemma \ref{lem_Fprop}, when $n\ge p+1$, $F(n,p)$ is a decreasing function in $n$ for a given $p$, where $F$ is as in \eqref{eqn_F}. In Table 1, we list the largest odd integer $n_p$ such that $F(n_p,p)\ge 1$ for each prime $p<45$ as follows.
\begin{table}[!ht]
	\centering
    \caption{The largest dimension $n_p$ not excluded by Theorem \ref{thm_moo} for $p<45$}
	\begin{tabular}{c|c|c|c|c|c|c|c|c|c|c|c|c|c|c}
		\hline
		$p$ & $2$ & $3$ & $5$ & $7$ & $11$ & $13$ & $17$ & $19$ & $23 $&$29 $& $31$ & $37$& $41$ & $43$  \\ \hline
		$n_p$ & $5$ & $5$ & $7$ & $7$& $9$ & $9$ & $11$ &$11$ & $13$& $15$ & $15$ & $17$& $17$ & $17$  \\
		\hline
	\end{tabular}
\end{table}
The existence of ovoids in $H(n,p^{2d})$ can not be excluded by Theorem \ref{thm_moo} when $n\le n_p$. For such a pair $(p,n)$ with $5\le n\le n_p$, we search for $4$-tuples $(m,s,k,d)$ such that $2nd=mks$, $s|p^{nd}+1$, and the bounds in Lemma \ref{lem_noovoideven} and Lemma \ref{lem_odd} are satisfied. In Table 2 we list the cases that survive, where $k=\frac{2nd}{ms}$ is omitted.
\begin{table}[!ht]
\centering
	\caption{The cases where all parameter restrictions are satisfied}
	\begin{tabular}{|c|c|c|c|c|c|c|}
		\hline
		$(n,p^d)$& $(5,2^2)$ & $(5,3^2)$ & $(9,11)$& $(7,13)$ & $ (9,17)$ & $(15,29)$   \\ \hline
		 $s$ & $5$ &  $10$ & $18$ & $14$ & $18$ & $30$  \\  \hline
		 $m$ & $1,2,4$ & $1,2$  &  $1$ & $1$ & $1$ & $1$ \\ \hline
	\end{tabular}\label{tab_2}
\end{table}
For the case $H(n,q^2)=H(5,2^4)$ or $H(5,3^4)$, an exhaustive search by Magma \cite{Magma} shows that there is no  transitive ovoid of the prescribed form. In all the remaining cases, we have $s=2n$, $m=1$, $k=1$ and $d=1$. In the proof of Lemma \ref{lem_odd}, we have shown that $|X|\ge q$ ($m=1$ is odd), where $X=\{x^{p^{ka}}:\,0\leq a\leq s-1\}$ for some element $x$ of $\F_{q^n}^*$. Since $q$ is a prime in these cases, we see that $|X|\le n$. We deduce that $n\ge p$, which is incorrect in each case. This excludes all the cases listed in Table \ref{tab_2}.
\end{proof}
\medskip

Theorem \ref{main} now follows by combining \cite{Bamberg2009Classificationovoid,Thas1981} and Theorem \ref{thm_main2}, Theorem \ref{thm_highdim}.\\

\noindent\textbf{Acknowledgement.} This work was supported by National Natural Science Foundation of China under Grant No. 11771392. The authors thank the anonymous reviewers for their detailed comments and suggestions that helped to improve the presentation of the paper.\\
\begin{center}
\scriptsize
\bibliographystyle{plain}
\footnotesize
\bibliography{Tovoid}

\begin{thebibliography}{10}

\bibitem{BEKS1993}
R.~D. Baker, G.~L. Ebert, G.~Korchm\'{a}ros, and T.~Sz\H{o}nyi.
\newblock Orthogonally divergent spreads of {H}ermitian curves.
\newblock In {\em Finite geometry and combinatorics ({D}einze, 1992)}, volume
  191 of {\em London Math. Soc. Lecture Note Ser.}, pages 17--30. Cambridge
  Univ. Press, Cambridge, 1993.

\bibitem{BambergKLP2007}
J.~Bamberg, S.~Kelly, M.~Law, and T.~Penttila.
\newblock Tight sets and {$m$}-ovoids of finite polar spaces.
\newblock {\em J. Combin. Theory Ser. A}, 114(7):1293--1314, 2007.

\bibitem{BT2006}
J.~Bamberg and T.~Penttila.
\newblock Transitive eggs.
\newblock {\em Innov. Incidence Geom.}, 4:1--12, 2006.

\bibitem{BT2008Overgroup}
J.~Bamberg and T.~Penttila.
\newblock Overgroups of cyclic {S}ylow subgroups of linear groups.
\newblock {\em Comm. Algebra}, 36(7):2503--2543, 2008.

\bibitem{Bamberg2009Classificationovoid}
J.~Bamberg and T.~Penttila.
\newblock A classification of transitive ovoids, spreads, and {$m$}-systems of
  polar spaces.
\newblock {\em Forum Math.}, 21(2):181--216, 2009.

\bibitem{BM1995}
A.~Blokhuis and G.~E. Moorhouse.
\newblock Some {$p$}-ranks related to orthogonal spaces.
\newblock {\em J. Algebraic Combin.}, 4(4):295--316, 1995.

\bibitem{Magma}
W.~Bosma, J.~Cannon, C~Fieker, and et~al.
\newblock Handbook of magma functions.
\newblock {\em Handbook of Magma Functions}, 2013.

\bibitem{BHR407}
J.~N. Bray, D.~F. Holt, and C.~M. Roney-Dougal.
\newblock {\em The maximal subgroups of the low-dimensional finite classical
  groups}, volume 407 of {\em London Mathematical Society Lecture Note Series}.
\newblock Cambridge University Press, Cambridge, 2013.
\newblock With a foreword by Martin Liebeck.

\bibitem{Brouwer1990}
A.~E. Brouwer and H.~A. Wilbrink.
\newblock Ovoids and fans in the generalized quadrangle {$Q(4,2)$}.
\newblock {\em Geom. Dedicata}, 36(1):121--124, 1990.

\bibitem{Cossidente2003Tovoid}
A.~Cossidente and G.~Korchm\'{a}ros.
\newblock Transitive ovoids of the {H}ermitian surface of {${\rm PG}(3,q^2)$},
  {$q$} even.
\newblock {\em J. Combin. Theory Ser. A}, 101(1):117--130, 2003.

\bibitem{CT2005}
A.~Cossidente and T.~Penttila.
\newblock Hemisystems on the {H}ermitian surface.
\newblock {\em J. London Math. Soc. (2)}, 72(3):731--741, 2005.

\bibitem{DeM2005}
J.~De~Beule and K.~Metsch.
\newblock The {H}ermitian variety {$H(5,4)$} has no ovoid.
\newblock {\em Bull. Belg. Math. Soc. Simon Stevin}, 12(5):727--733, 2005.

\bibitem{GPPS1999}
R.~Guralnick, T.~Penttila, C.~E. Praeger, and J.~Saxl.
\newblock Linear groups with orders having certain large prime divisors.
\newblock {\em Proc. London Math. Soc. (3)}, 78(1):167--214, 1999.

\bibitem{GGG}
J.~W.~P. Hirschfeld and J.~A. Thas.
\newblock {\em General {G}alois geometries}.
\newblock Springer Monographs in Mathematics. Springer, London, 2016.

\bibitem{KL129}
P.~Kleidman and M.~Liebeck.
\newblock {\em The subgroup structure of the finite classical groups}, volume
  129 of {\em London Mathematical Society Lecture Note Series}.
\newblock Cambridge University Press, Cambridge, 1990.

\bibitem{LidlFF}
R.~Lidl and H.~Niederreiter.
\newblock {\em Finite fields}, volume~20 of {\em Encyclopedia of Mathematics
  and its Applications}.
\newblock Cambridge University Press, Cambridge, second edition, 1997.
\newblock With a foreword by P. M. Cohn.

\bibitem{Moor1996}
G.~E. Moorhouse.
\newblock Some {$p$}-ranks related to {H}ermitian varieties.
\newblock volume~56, pages 229--241. 1996.
\newblock Special issue on orthogonal arrays and affine designs, Part II.

\bibitem{Stirling1955}
H.~Robbins.
\newblock A remark on {S}tirling's formula.
\newblock {\em Amer. Math. Monthly}, 62:26--29, 1955.

\bibitem{Thas1981}
J.~A. Thas.
\newblock Ovoids and spreads of finite classical polar spaces.
\newblock {\em Geom. Dedicata}, 10(1-4):135--143, 1981.

\bibitem{Thas2001}
J.~A. Thas.
\newblock Ovoids, spreads and {$m$}-systems of finite classical polar spaces.
\newblock In {\em Surveys in combinatorics, 2001 ({S}ussex)}, volume 288 of
  {\em London Math. Soc. Lecture Note Ser.}, pages 241--267. Cambridge Univ.
  Press, Cambridge, 2001.

\bibitem{Thas1994Spreads}
J.~A. Thas and S.~E. Payne.
\newblock Spreads and ovoids in finite generalized quadrangles.
\newblock {\em Geom. Dedicata}, 52(3):227--253, 1994.

\bibitem{Z1892}
K.~Zsigmondy.
\newblock Zur {T}heorie der {P}otenzreste.
\newblock {\em Monatsh. Math. Phys.}, 3(1):265--284, 1892.

\end{thebibliography}
\end{center}
\begin{flushleft}
 Tao Feng, Weicong Li\\\vspace*{2mm}

 School of Mathematical Sciences, \\
Zhejiang University, 38 Zheda Road, \\
 Hangzhou 310027, Zhejiang P.R. China,\\\vspace*{2mm}

 Email: tfeng@zju.edu.cn, conglw@zju.edu.cn
\end{flushleft}
\end{document}